\documentclass[a4paper,12pt]{amsart}%
\usepackage{eurosym}
\usepackage{amsmath}
\usepackage{mathrsfs}
\usepackage{amsfonts}
\usepackage{graphicx}
\usepackage{amsfonts}
\usepackage{amssymb}%
\usepackage{amsthm}
\usepackage[normalem]{ulem}
\usepackage{esint}
\usepackage{color}
\usepackage{float}
\usepackage{enumitem}
\usepackage{amsmath}
\usepackage{mathtools}
\usepackage[abbrev,backrefs]{amsrefs}

\usepackage{palatino}

\usepackage{tikz}
\usetikzlibrary{positioning}

\makeatletter
\newcases{PDEcases}{\quad}{%
  \hfil$\m@th\displaystyle{##}$}{{##}\hfil}{\lbrace}{.}
\makeatother

\providecommand{\U}[1]{\protect\rule{.1in}{.1in}}

\newtheorem{theorem}{Theorem}[section]

\newtheorem{lemma}[theorem]{Lemma}

\newtheorem{definition}[theorem]{Definition}

\newtheorem{remark}[theorem]{Remark}
\newtheorem{corollary}[theorem]{Corollary}

\numberwithin{equation}{section}

\newcommand{\R}{\mathbb{R}}

\title[Estimates for solutions of critical elliptic equations]{$BMO$ and gradient estimates for solutions of critical elliptic equations}

\begin{document}

\author[Y.-W. Chen]{You-Wei Benson Chen}
\author[J. Manfredi]{Juan Manfredi}
\author[D. Spector]{Daniel Spector}

\address[Y.-W. Chen]{
National Taiwan University, Department of
Mathematics, Taipei 10617, TAIWAN.
}
\email{bensonchen.sc07@nycu.edu.tw}

\address[J. Manfredi]{Department of Mathematics,
University of Pittsburgh,
 Pittsburgh, PA 15260, USA}
\email{manfredi@pitt.edu}

\address[D. Spector]{Department of Mathematics, National Taiwan Normal University, No. 88, Section 4, Tingzhou Road, Wenshan District, Taipei City, Taiwan 116, R.O.C.
\newline
National Center for Theoretical Sciences\\No. 1 Sec. 4 Roosevelt Rd., National Taiwan
University\\Taipei, 106, Taiwan}
\email{spectda@protonmail.com}

\begin{abstract}
    In this paper we explore several applications of the recently introduced spaces of functions of bounded $\beta$-dimensional mean oscillation for $\beta \in (0,n]$ to regularity theory of critical exponent elliptic equations. \par
       We first show that functions with gradient in weak-$L^n$ are in $BMO^\beta$ for any $\beta \in (0,n]$, improving the classical result $\nabla u\in L^n$ implies $u\in BMO$. We apply this result to the Poisson equation $-\Delta u = \operatorname*{div} F$ with zero boundary conditions in a bounded $C^1$ domain to show that $u\in BMO^{\beta}$ when $F$ is in weak-$L^n$. Next, we consider the $n$-Laplace equation
    \begin{equation*}
    \left\{
\begin{split}
-\operatorname*{div}( |\nabla U|^{n-2} \nabla U) &= F  \mbox{ in }  \Omega,\\
 U  &=0   \mbox{ on }\partial \Omega.
\end{split}\right.
\end{equation*} with $F\in L^1(\Omega)$
and show that the classical result $u\in BMO$ can be improved to $u\in BMO^\beta$. Finally, we  consider the $n$-Laplace equation in the case when $F \in L^1$, $\operatorname*{div} F=0$  and prove that for smooth domains  $\Omega$ we have the estimate   \begin{align*}
\|\nabla U \|_{L^n} \mathbb \leq C \, \|F\|^{1/(n-1)}_{L^1},
\end{align*}
where the constant $C$ is independent of $F$. 
\end{abstract}
\maketitle
\section{Introduction and Main Results}

The space of functions of bounded mean oscillation appears readily in harmonic analysis in the study of the boundedness of translation invariant singular integral operators \cite{Peetre}*{Theorem 1.1 and Remark 1.3}, as an endpoint in interpolation \cite{FeffermanStein}*{Section III}, in PDE in the treatment of regularity theory of elliptic equations \cite{Moser1960}, and in Sobolev spaces in the consideration of critical exponent embeddings \cites{Adams1973, Cianchi, MS ,Trudinger, Yudovich}.

While its replacement of $L^\infty$ in harmonic analysis seems sharp, in the study of Sobolev embeddings one can obtain improvements.  In particular, from the trace theory one understands that Sobolev functions in the critical exponent are actually of bounded mean oscillation along hypersurfaces of any dimension $k=1,\ldots,n$:

When $kp=n$, 
\begin{align*}
    W^{k,p}(Q) \hookrightarrow BMO(Q).
\end{align*}
Yet a stronger conclusion can be deduced: 
 For a $n-1$ dimensional hyperplane $H$, trace theory and the Sobolev embedding applied to the trace yields
\begin{align*}
    W^{k,p}(Q) \hookrightarrow W^{k-1/p,p}(Q \cap H) \hookrightarrow BMO(Q \cap H).
\end{align*}

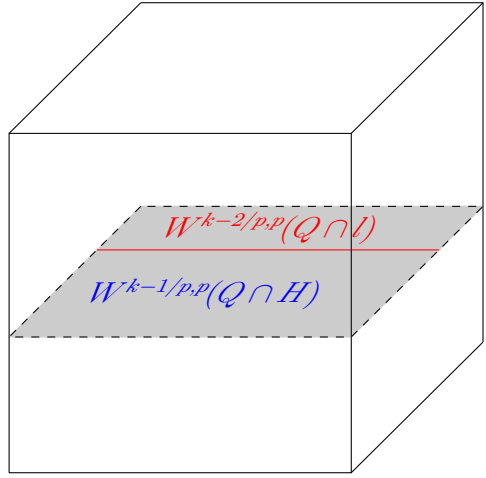
\vspace{.5cm}

\hspace{2cm}
\begin{tikzpicture}

    \def\sliceZ{1.8}
    \def\side{4.5}
    \def\zangle{65}

    \filldraw[color=gray!40] (0,\sliceZ,0) -- (0,\sliceZ,\side) -- (\side,\sliceZ,\side) -- (\side,\sliceZ,0) -- cycle;
    \draw[dashed] (0,\sliceZ,0) -- (0,\sliceZ,\side) -- (\side,\sliceZ,\side) -- (\side,\sliceZ,0) -- cycle;
 \draw[red] (0,\sliceZ,1.5) -- (\side,\sliceZ,1.5);
    
    \draw (\side,0,0) -- (\side,\side,0) node[midway,right] {} -- (0,\side,0);
    \draw (0,0,\side) -- (\side,0,\side) node[midway,below] {} -- (\side,\side,\side) -- (0,\side,\side) -- (0,0,\side);
    \draw (\side,0,0) -- (\side,0,\side) node[midway,below right] {};
    \draw (\side,\side,0) -- (\side,\side,\side);
    \draw (0,\side,0) -- (0,\side,\side);

    \node[cm={1,0,cos(\zangle),sin(\zangle),(0,0)}] at (2,\sliceZ,3){\color{blue}$W^{k-1/p,p}(Q \cap H)$};
    \node[cm={1,0,cos(\zangle),sin(\zangle),(0,0)}] at (2,\sliceZ,0.75){\color{red}$W^{k-2/p,p}(Q \cap l)$};

\end{tikzpicture}
\noindent

Iteration of this trace and embedding argument can be continued until one reaches a line $l$:
\begin{align*}
    W^{k,p}(Q) \hookrightarrow W^{k-1/p,p}(Q \cap H) \hookrightarrow \cdots \hookrightarrow W^{k-(n-1)/p,p}(Q \cap l) \hookrightarrow BMO(Q \cap l),
\end{align*}
and thus one sees that Sobolev functions in the critical exponent restrict as functions of bounded mean oscillation on lower dimensional hypersurfaces down to dimension one.

Meanwhile, from the perspective of trace integrability it was progressively understood  \cites{Adams1973,Cianchi,MS, Yudovich} that such functions are
exponentially integrable with respect to measures in the Morrey space $\mathcal{M}^\beta(Q)$  for any $\beta \in (0,n]$:  There exist $c_\beta,C_\beta>0$ such that
\begin{align*}
		\int_Q \exp\left(c_\beta \left|\frac{u(x)}{\| \nabla^k u\|_{L^p}}\right|^{p'}\right) \;d\mu \leq C_\beta
	\end{align*}
	for all $u \in W^{k,p}_0(Q)$ and $\mu$ is any locally finite Radon measure which satisfies the ball growth condition 
\begin{align*}
\|\mu\|_{\mathcal{M}^\beta(Q)} := \sup_{r>0,x\in \mathbb{R}^n} \frac{|\mu|(B(x,r) \cap Q)}{r^\beta} <+\infty.
\end{align*}

The improved properties of these special functions of bounded mean oscillation led the first and third authors in the recent paper \cite{Chen-Spector} to introduce spaces that capture these phenomena
 -- the so-called spaces of bounded $\beta$-dimensional mean oscillation $BMO^\beta(\mathbb{R}^n)$.  These spaces are refinements of the classical space of functions of bounded mean oscillation in the sense that 
\begin{enumerate}
    \item  If $0<\alpha< \beta \leq n$
\begin{align*}
    BMO^\alpha(\mathbb{R}^n) \subset BMO^\beta(\mathbb{R}^n) \subset BMO^n(\mathbb{R}^n) \equiv BMO(\mathbb{R}^n)
\end{align*}
\item   If $u \in BMO^\beta(\mathbb{R}^n)$, then for any $k\geq \beta$, $u|_H \in BMO(\mathbb{R}^k)$ for any $k$ dimensional hyperplane $H$.
\item We have the John-Nirenberg inequality
\begin{align*}
\mathcal{H}^{\beta}_\infty\left(\{x\in Q:|u(x)-c_Q|>t\}\right) \leq C\, l(Q)^\beta \exp(-c\,t/\|u\|_{BMO^\beta(\mathbb{R}^n)}).
\end{align*}
\end{enumerate}
For precise definitions of the Hausdorff content and $BMO^\beta(\mathbb{R}^n)$, we refer the reader to Section \ref{preliminaries}.

The goal of the present paper is to show how these spaces can give improvements to certain estimates for PDE. To this end, let $\Omega \subset \mathbb{R}^n$ be a $C^1$-domain,  suppose $F \in L^n(\Omega;\mathbb{R}^n)$,  and consider the problem
\begin{equation}\label{div_poisson}\left\{
\begin{array}{rccl}
-\Delta u  &=&  \operatorname*{div}F & \mbox{ in }  \Omega,\\
 u  &= &0   &\mbox{ on }\partial \Omega.
\end{array} \right.
\end{equation}

The classical theory for this equation is the transfer of $L^p$ regularity from $F$ to $\nabla u$:  For $1<p<+\infty$, one has
\begin{align}\label{elliptic_regularity}
\Vert \nabla u \Vert_{L^{p}(\Omega;\mathbb{R}^n) }  \leq C \Vert F\Vert_{L^{p}(\Omega;\mathbb{R}^n) }, 
  \end{align}
see e.g. \cite[Theorem 1]{AQ2002} or \cite{Byun}*{Theorem 1.5}.  Therefore when $F \in L^n(\Omega;\mathbb{R}^n)$, one has $\nabla u  \in L^n(\Omega;\mathbb{R}^n)$ from which one deduces $u \in BMO(\Omega)$.  Our first result is a refinement of this estimate for spaces $BMO^\beta(\Omega)$, which we define in analogy with those introduced in \cite{Chen-Spector}, but now only taken into account the behavior on a bounded Lipschitz domain $\Omega \subset \mathbb{R}^n$:  For a scalar function $u \in L^1(\Omega;\mathcal{H}^\beta_\infty)$, we introduce the quasi-norm
\begin{align*}
\|u\|_{BMO^{\beta}(\Omega)} \vcentcolon=\sup_{2Q \subset \Omega} \inf_{c \in \mathbb{R}} \frac{1}{l(Q)^\beta}  \int_{Q}  |u-c| \;d\mathcal{H}^{\beta}_\infty,
\end{align*}
where the supremum is taken over all finite subcubes $Q \subseteq \Omega$.  Whether one takes the supremum over cubes or balls is equivalent, provided they stay away from the boundary.  The definition for vector-valued functions is defined analogously, componentwise.  We then define the space
\begin{align*}
BMO^\beta(\Omega):= \left\{ u \in L^1_{loc}(\Omega;\mathcal{H}^\beta_\infty) \vcentcolon  \|u\|_{BMO^{\beta}(\Omega)}<+\infty \right\}.
\end{align*}
Actually, the definition of $BMO^\beta(\Omega)$ make sense for an arbitrary open set, though with a focus on applications to PDE, in this paper we always  assume $\Omega$ is as smooth as needed. 

With this definition, an adaptation of the arguments in \cite{Chen-Spector} to the setting of a bounded domain would yield that for $u$ which satisfies \eqref{div_poisson} one has  $u \in BMO^{n-1+\epsilon}(\Omega)$ for any $\epsilon>0$.  However, as discussed above one expects at least an embedding into $BMO^1(\Omega)$ from the trace theory.  The first result of this paper is the following theorem confirming this intuition, that such functions are of bounded $\beta$-dimensional mean oscillation for any $\beta>0$.

\begin{theorem}\label{poisson_thm}
 Let $\Omega$ be a bounded $C^{1}$-domain.  Suppose that $u$ satisfies \eqref{div_poisson} and $F \in L^{n,\infty}(\Omega;\mathbb{R}^n)$.  Then $u \in BMO^\beta(\Omega)$ for any $\beta \in (0,n]$.
\end{theorem}

Observe that Theorem \ref{poisson_thm} only requires $F \in L^{n,\infty}(\Omega;\mathbb{R}^n)$, the Marcinkiewicz space weak-$L^n$.  That such an assumption is sufficient follows in part by interpolation of \eqref{elliptic_regularity} to deduce that the solution of \eqref{div_poisson} for such $F$ admits the estimate
\begin{align*}
\Vert \nabla u \Vert_{L^{n,\infty}(\Omega;\mathbb{R}^n) }  \leq C \Vert F\Vert_{L^{n,\infty}(\Omega;\mathbb{R}^n) },
 \end{align*}
and from our
\begin{theorem}\label{BMObetabyPoincare}
    Let $0 < \beta \leq n \in \mathbb{N}$ and $\Omega \subseteq \mathbb{R}^n$ be an open set. 
 There exists a constant $C=C(\beta,n)>0$ such that 
   \begin{align}
       \Vert u \Vert_{BMO^\beta(\Omega) } \leq C \Vert \nabla u \Vert_{L^{n, \infty} (\Omega;\mathbb{R}^n)}
   \end{align}
   for all $u \in W^{1,1}_{loc}(\Omega)$ such that $\nabla u \in L^{n, \infty} (\Omega;\mathbb{R}^n)$.
\end{theorem}

\begin{remark}
The conclusion of Theorem \ref{poisson_thm} holds in a greater generality than simply the Poisson equation, for example, for linear equations which admit $L^p$ elliptic regularity estimates:  For any $1<p<+\infty$, there exist $\delta$ small and $R>0$ such that if $A$ is uniformly elliptic, $(\delta,R)$-vanishing $A$ and $\Omega$ a $(\delta,R)$-Lipschitz domain, as the solution to
\begin{equation}\label{div_A}
 \begin{split}
-\operatorname*{div}A(x)\nabla u & =  \operatorname*{div}F  \mbox{ in } \Omega,\\
 u  &= 0   \mbox{ on } \partial \Omega,
\end{split} 
\end{equation}
admits the estimate \eqref{elliptic_regularity}, see e.g. \cite{Byun}*{Theorem 1.5}. Notice that the hypothesis on $\Omega$ hold when
 $\Omega$ is a bounded $C^{1}$-domain. 
\end{remark}

The passage from $L^{n}(\Omega;\mathbb{R}^n)$ to $L^{n,\infty}(\Omega;\mathbb{R}^n)$ in Theorem \ref{BMObetabyPoincare} and the known inclusions
\begin{align*}
L^n(\Omega) \subset L^{n,\infty}(\Omega) \subset \mathcal{M}^{n-1}(\Omega)    
\end{align*}
suggest that it should be interesting to study the case $u \in W^{1,1}_{loc}(\Omega)$ is such that $\nabla u \in \mathcal{M}^{n-1}(\Omega)$, see e.g.~N. Trudinger's exponential integrability results under these assumptions
\cite{Trudinger}*{Theorem 1 on p.~476}. While Theorem \ref{BMObetabyPoincare} shows that even for functions with derivatives in the weak $L^n$ one has an embedding into $BMO^\beta$ for arbitrarily small $\beta>0$,  this is not the case of $\nabla u \in \mathcal{M}^{n-1}$.  In particular, in \cite{Chen-Spector} it was shown that one could have any $\beta>n-1,$ while the embedding fails for $\beta<n-1$.  We here recover the endpoint in the following
\begin{theorem}\label{MorreyBMO}
   There exists a constant $C= C(n)$ such that for  $u \in W^{1,1}_{loc}$ we have
\begin{align}\label{Poincare1}
    \inf_{c \in \mathbb{R}} \int_B |u-c|\; d\mathcal{H}^{n-1}_\infty \leq C \int_B |\nabla u | \; dx 
\end{align}
holds for any open ball $B\subseteq \Omega$.  Moreover, if $|\nabla u | \in \mathcal{M}^{n-1}(\Omega)$, then
\begin{align}\label{Morreyn1eestimate}
    \Vert u \Vert_{BMO^{n-1} (\Omega)} \leq C \Vert \nabla u \Vert_{\mathcal{M}^{n-1}(\Omega)}.
\end{align}
Here, $u$ on the left hand side is understood as its precise representative (see (\ref{preciserepresentative}) for definition).
\end{theorem}

That one has such a result in the first order case suggests investigation of the higher order analogue.  In this regime we prove
\begin{theorem}\label{MorreyBMO2}
    Let $\Omega$ be an open bounded set in $\mathbb{R}^n$, $1 \leq k \leq n-1$ and $u \in W^{k,1}_{0} (\Omega)$. If $| \nabla^j u| \in \mathcal{M}^{n-j} (\Omega)$ for $1 \leq j  \leq k$, then there exists a constant $C$ such that 
\begin{align}\label{BMO2estimate}
    \Vert u \Vert_{BMO^{n-k}(\Omega)} \leq C \sum_{j=1}^{k} \Vert \nabla^j u \Vert_{\mathcal{M}^{n-j}(\Omega)},
\end{align}
where $u$ on the left hand side is understood as its precise representative (see (\ref{preciserepresentative}) for definition).
\end{theorem}

We next turn our attention to the critical regime in the nonlinear setting.  In particular, let $F \in L^1(\Omega;\mathbb{R}^n)$ and consider the system
\begin{equation}\label{div_poisson1}
\begin{split}
-\operatorname*{div}( |\nabla U|^{n-2} \nabla U) &= F  \mbox{ in }  \Omega,\\
 U  &=0   \mbox{ on }\partial \Omega.
\end{split}
\end{equation}
The notion of solution for this type of system needs to be defined with care since $\int_{\Omega} F\cdot \Phi \, dx$ is not well defined for $\Phi\in W^{1,n}(\Omega; \mathbb{R}^n)$. A suitable definition of solution can be found in  the work of Dolzmann, Hungerb\"{u}hler, and M\"{u}ller \cite[Definition 1 on p.~546]{DHM1}, whose results \cite{DHM, DHM1} include existence and uniqueness of solutions on bounded Lipschitz domains  satisfying  $U \in BMO(\Omega)$ and $\nabla U \in  L^{n, \infty} (\Omega; \mathbb{R}^{n \times n})$.  Their results hold even in the more general setting $F \in M_b(\Omega;\mathbb{R}^n)$, the space of finite vector-valued Radon measures on $\Omega$, though we restrict our consideration to $L^1(\Omega;\mathbb{R}^n)$ in the sequel.

A combination of the results of \cite{DHM, DHM1} and our Theorem \ref{BMObetabyPoincare} yields
\begin{corollary}\label{BMObetaforpLapacian}
Let $\Omega$ be a bounded Lipschitz domain. Let $F \in L^1(\Omega;\mathbb{R}^n)$ and suppose $U$ satisfies \eqref{div_poisson1}.  Then $U \in BMO^\beta(\Omega)$ for any $\beta \in (0,n]$ and we have the estimate
\begin{align}
\| U \|_{BMO^\beta(\Omega)} \mathbb \leq \, C \, \|F\|_{L^1(\Omega )},
\end{align}
where the constant $C=C(\beta, n)$ depends only on $\beta$ and $n$. 
\end{corollary}
\noindent
For existence and uniqueness for \eqref{div_poisson1} in a slightly different class of Sobolev spaces, see the the work of Greco, Iwaniec, and Sbordone \cite{GIS}.

The consideration of this system in light of the work of Bourgain and Brezis on regularity for elliptic systems with $L^1$ data \cites{BourgainBrezis2004,BourgainBrezis2007} suggests it should be interesting to study a constrained system.  In particular, in the vector case in the linear setting ($n=2$), under the assumption that $\operatorname*{div} F=0$, Bourgain and Brezis had shown that \eqref{div_poisson1} admits a solution for which $\nabla U \in L^2$.  We here observe that this result extends to the non-linear setting in

 \begin{theorem}\label{estimatefornonlinearpLaplacian} Let $\Omega$ be a  smooth bounded domain.  Let $F \in L^1(\Omega;\mathbb{R}^n)$ with $\operatorname*{div}F=0$.  Then there exists a unique weak solution $U \in W^{1,n}_0(\Omega;\mathbb{R}^{n})$ of the system
\begin{equation}\label{div_n_laplace}\left\{
\begin{split}
-\operatorname*{div}( |\nabla U|^{n-2} \nabla U) &= F  \mbox{ in }  \Omega,\\
 U  &=0   \mbox{ on }\partial \Omega.
\end{split}\right.
\end{equation}
Moreover, one has the estimate
\begin{align}
\|\nabla U \|_{L^n(\Omega ; \mathbb{R}^{n \times n})} \mathbb \leq C \|F\|^{1/(n-1)}_{L^1(\Omega ; \mathbb{R}^{n})},
\end{align}
where the constant $C$ is independent of $F$. 
\end{theorem}

We first obtained Theorem \ref{estimatefornonlinearpLaplacian} by a combination of arguments involving $BMO^1(\Omega)$ and the atomic decomposition established in \cite{HS} (see also \cites{GHS,HRS,RSS}).  After further study, we realized that in fact it is an easy consequence of the work of Brezis and Van Schaftingen \cite{BVS}, and in fact their result can be applied generally to a large class of non-linear equations:

\begin{theorem}\label{estimatefornonlinearpLaplacian_A}
Let $\Omega$ be a smooth bounded domain.
 Suppose $\mathcal{A}: \Omega \times \mathbb{R}^{n\times n}\mapsto\mathbb{R}^n$ is a Carath\'eodory function; that is,  for $x\in\Omega$, the mapping $\xi\mapsto\mathcal{A}(x,\xi)$ is measurable, for a.e. $x$, the mapping
$x\mapsto\mathcal{A}(x,\xi)$ is continuous, there exists a constant $c_1>0$ such that
\begin{align*}
c_1\,|\xi|^n \leq \mathcal{A}(x,\xi)\cdot \xi,
\end{align*}
and there exists a constant $c_2>0$ such that
\begin{align*}
| \mathcal{A}(x,\xi)| \le c_2\, | \xi|^{n-1}.
\end{align*}
Let $F \in L^1(\Omega;\mathbb{R}^n)$ with $\operatorname*{div}F=0$.  Then there exists a unique weak solution $U \in W^{1,n}_0(\Omega;\mathbb{R}^{n})$ of the system
\begin{equation}\label{div_n_laplace_A0}
\begin{split}
-\operatorname*{div} \mathcal{A}(x,\nabla U)&= F  \mbox{ in }  \Omega,\\
 U  &=0   \mbox{ on }\partial \Omega.
\end{split}
\end{equation}
Moreover, one has the estimate
\begin{align}
\|\nabla U \|_{L^n(\Omega;\mathbb{R}^{n\times n})} \leq C \|F\|^{1/(n-1)}_{L^1(\Omega;\mathbb{R}^n)},
\end{align}
where the constant $C$ is independent of $F$. 
\end{theorem}
\noindent
Note that the hypothesis on $\mathcal{A}(x,\xi)$ imply that it satisfies the hypothesis of Theorem 1.1 (Existence) and
Theorem 1.2 (Uniqueness) from \cite{DHM,DHM1}. 

Theorems \ref{estimatefornonlinearpLaplacian} and \ref{estimatefornonlinearpLaplacian_A} prompt one to wonder whether their conclusions can be strengthened on the Lorentz scale.  In particular, we recall Bourgain and Brezis' observation that for $n=2$, one has $\nabla U \in L^{2,1}$ (the Lorentz space) and therefore $U \in L^\infty$.  Indeed, we have the following analogous result for the local solutions of \eqref{div_poisson1}:
\begin{theorem}\label{estimatefornonlinearpLaplacian_Lorentz}
Let $\Omega$ be a  smooth bounded domain.  Suppose $F \in L^1(\Omega;\mathbb{R}^n)$ is such that $\operatorname*{supp} F \subset \Omega$ and $\operatorname*{div}F=0$ and suppose $U$ is the solution  of
\begin{equation}\label{div_n_laplace2}\left\{
\begin{split}
-\operatorname*{div}( |\nabla U|^{n-2} \nabla U) &= F  \mbox{ in }  \Omega,\\
 U  &=0   \mbox{ on }\partial \Omega.
\end{split}\right.
\end{equation}
Then $\nabla U \in L^{n,n-1}_{loc} (\Omega ; \mathbb{R}^{n \times n})$.  In particular, for every ball $B_1 \subset \Omega$, denoting by $B_\theta$ the ball with same center and radius $\theta \in (0,1)$, we have
\begin{align}
\|\nabla U \|_{L^{n,n-1}(B_\theta;\mathbb{R}^{n\times n}))} \leq C \|F\|^{1/(n-1)}_{L^1(\Omega;\mathbb{R}^n)},
\end{align}
where the constant $C$ is independent of $F$. 
\end{theorem}

Here we rely on a recent result of Martino and Schikorra \cite{MS2024}
which shows that for $G \in L^{n/(n-1),1}(B_1;\mathbb{R}^n)$, any $U$ which satisfies
\begin{equation}\label{MS}
-\operatorname*{div}( |\nabla U|^{n-2} \nabla U) = \operatorname*{div} G  \text{ in }  B_1
\end{equation}
is such that $\nabla U \in L^{n,n-1}(B_\theta;\mathbb{R}^n)$ for some $\theta>0$.  Moreover, this result is optimal in that the second exponent cannot be improved.

We conclude the introduce with a few remarks about the role of $p$, $n$, the divergence free condition, and regularity of the boundary of the domain.  In Theorem \ref{poisson_thm}, Theorem \ref{BMObetabyPoincare}, and Corollary \ref{BMObetaforpLapacian}, the fact that the integrability exponent $p$ is equal to the dimension of the ambient space $n$ is precisely because this is the critical regime of the Sobolev embedding.  This criticality is also present in Theorems \ref{MorreyBMO} and \ref{MorreyBMO2}, though its relation with the preceding criticality is not obvious from the definition of the spaces.  This is not the case for Theorems \ref{estimatefornonlinearpLaplacian}, \ref{estimatefornonlinearpLaplacian_A}, and \ref{estimatefornonlinearpLaplacian_Lorentz}, and it would be interesting to know if these improvements from weak-type to strong-type estimates when the right-hand-side is divergence free, or more generally following the discussion below satisfies a cocancelling differential constraint/is in a dimension stable space of measures, extend to the $p$-Laplace equation for $p\not \in  \{2,n\}$ or more general equations.  

The choice to present Theorems \ref{estimatefornonlinearpLaplacian}, \ref{estimatefornonlinearpLaplacian_A}, and \ref{estimatefornonlinearpLaplacian_Lorentz} for vector-valued functions $U,F: \mathbb{R}^n \to \mathbb{R}^n$ stems from the reliance on Brezis and Van Schaftingen's estimates in \cite{BVS} and the $L^1$ Sobolev embeddings of the third author in \cite{HS,HRS}.  The point here is that such estimates do not hold for unconstrained fields $F$.  However, if one considers the differential equations on $\mathbb{R}^n$ without boundary conditions, the conclusions are valid for any $F:\mathbb{R}^n\to \mathbb{R}^k$ for which $L F=0$ for some homogeneous linear differential operator $L(D):C^\infty_c(\mathbb{R}^d;\mathbb{R}^k) \to C^\infty_c(\mathbb{R}^d;\mathbb{R}^l)$ which is
cocancelling \cite[Definition 1.3 on p.~881]{VS3}, i.e. 
\begin{align*}
\bigcap_{\xi\in\mathbb{R}^d \setminus \{0\}}\ker L(\xi)=\{0\}.
\end{align*}
Here one utilizes the embeddings proved by Van Schaftingen \cite[Theorem 1.4 on p.~881]{VS3} for the Lebesgue scale estimate and D. Stolyarov \cite[Theorem 1 on p.~844]{Stolyarov} for the Lorentz scale estimate in place of \cite{BVS} and \cite{HS,HRS}.  If one wishes to dispense with the vector setting entirely, in place of differential constraints one can assume $F$ is in a dimension stable space of measures introduced in \cite{DS}.  It is possible these results continue to hold with boundary conditions, though this requires some care about their imposition and the regularity of the boundary of the domain.  

Our assumptions on the regularity of the boundary in this paper are such that one can use as blackboxes the results of \cite{Byun} for $L^p$ theory for the Poisson equation, \cite{DHM,DHM1} for existence/uniqueness for the $n$-Laplace, and \cite{BVS} for a $W^{1,n}$ duality estimate.  The weakest of these assumptions are those in \cite{DHM,DHM1}, which are valid for open bounded sets with a density condition on their complement. It is believable that the estimates for $U$ in $BMO^\beta(\Omega)$ are valid under this condition, though this requires a more delicate argument, as even a Lipschitz boundary is not sufficient for $L^p$ theory for the Poisson equation when $n\geq 4$, see e.g. \cite[Theorem 1.2(a) on p.~167]{JK}.

\section{Preliminaries}\label{preliminaries}
In this section, we establish the notational framework and foundational definitions pertinent to the Choquet integral with respect to Hausdorff content $ \mathcal{H}^\beta_\infty$. This groundwork is essential for discussing the concept of $BMO^\beta$ introduced in \cite{Chen-Spector}. \par

For $0< \beta \leq n \in \mathbb{N}$, the Hausdorff content $\mathcal{H}^{\beta}_{\infty}$ of a subset $E \subseteq \mathbb{R}^n$ is defined as
\begin{align*}
\mathcal{H}^{\beta}_{\infty}(E)\vcentcolon= \inf \left\{\sum_{i=1}^\infty \omega_\beta r_i^\beta : E \subset \bigcup_{i=1}^\infty B(x_i,r_i) \right\},
\end{align*}
where $\omega_\beta\vcentcolon= \pi^{\beta/2}/\Gamma(\beta/2+1)$ is a normalization constant.

Let $\Omega$ be an open set in $\mathbb{R}^n$. We recall that the Choquet integral of $f : \Omega \subset \mathbb{R}^n \to [0,\infty]$ is defined by
\begin{align}\label{choquet}
\int_{\Omega} f\;d\mathcal{H}^{\beta}_\infty\vcentcolon= \int_0^\infty \mathcal{H}^{\beta}_\infty\left(\left\{x \in \Omega: f(x)>t\right\}\right)\;dt.
\end{align}
We next introduce the vector space $L^p(\Omega;\mathcal{H}^{\beta}_\infty)$, which consists of all $\mathcal{H}^{\beta}_\infty$-quasicontinuous functions (see \cite[Section
2]{Chen-Spector} or \cite[Section
3]{PonceSpector}) $f$ which satisfy
\begin{align}\label{norm}
\|f\|_{L^p(\Omega;\mathcal{H}^{\beta}_\infty)}\vcentcolon= \left(\int_{\Omega} |f|^{p} \;d\mathcal{H}^{\beta}_\infty\right)^\frac{1}{p}<+\infty.
\end{align}
For convenience, in the case where $ \Omega = \mathbb{R}^n $, we adopt the notation
\begin{align*}
\|f\|_{L^p(\mathcal{H}^{\beta}_\infty)} \coloneqq \left( \int_{\mathbb{R}^n} |f|^p \, d\mathcal{H}^{\beta}_\infty \right)^\frac{1}{p} < +\infty.
\end{align*}

It can be shown that \eqref{norm} is a quasi-norm, though there exists a norm equivalent to \eqref{norm} such that $L^p(\Omega;\mathcal{H}^{\beta}_\infty)$ equipped with this equivalent norm is a Banach space. 
For $p =1$, this fact is written in \cite[Theorem 2.8]{Chen-Spector}. For $p>1$ one then follows the standard arguments in measure theory. Additionally, by Proposition 1 in \cite{Adams1988} we have the topological dual of $L^1(\Omega;\mathcal{H}^{\beta}_\infty)$ is the Morrey space
\begin{align*}
\mathcal{M}^{\beta}(\Omega)\vcentcolon= \left\{ \nu \in M_{loc}(\Omega): \|\mu\|_{\mathcal{M}^\beta (\Omega)} \vcentcolon= \sup_{x \in \Omega,r>0} \frac{|\mu|(B(x,r) \cap \Omega)}{\;r^\beta}<\infty \right\},
\end{align*}
where $M_{loc}(\Omega)$ is the set of locally finite Radon measures in $\Omega$ (see \cite{Adams1988}*{p.~118}).  It then follows from this duality that
\begin{align}\label{HB}
\int_{\Omega} |f|\;d\mathcal{H}^{\beta}_\infty \asymp \sup_{\|\nu\|_{\mathcal{M}^\beta (\Omega)} \leq 1}  \left|\int_{\Omega} f \;d\nu \right|
\end{align}
for $f \in L^1(\Omega;\mathcal{H}^{\beta}_\infty)$. Moreover, we adopt the convention that a locally integrable function $f :\mathbb{R}^n \to \mathbb{R}$ is said to be in Morrey space $\mathcal{M}^\beta(\Omega)$ provided
\begin{align*}
 \Vert f \Vert_{\mathcal{M}^\beta(\Omega)} :=   \sup_{x \in \Omega,r>0} \int_{\Omega \cap B(x,r)} |f|\; dx < \infty
\end{align*}

We note that if $f\in W^{1,p} (\Omega)$, then $f$ is considered identical to $g \in W^{1,p} (\Omega)$, provided $f(x) = g(x)$ for Lebesgue almost every $x \in \Omega$. Therefore, it is necessary to consider the precise representative of functions in the Sobolev space if we want to establish bounds for $\int_{\Omega } |f| \; d\mathcal{H}^\beta_\infty$ when $f$ is in $W^{1,p} (\Omega)$. The notation we use is based on \cite{EvansGariepy} as follows.
\begin{definition}
    Let $\Omega$ be an open set in $\mathbb{R}^n$ and $u\in L^1_{loc} (\Omega)$. Then the precise representative $u^*$ of $u$ is defined as 
 \begin{equation}\label{preciserepresentative}
    u^*(x) =
    \begin{cases*}
      \lim\limits_{r \to 0} \frac{1}{r^n} \int_{B_r(x)} u(y) \; dy & if x \text{ is a Lebesgue point of } u  \\
      0        & otherwise.
    \end{cases*}
  \end{equation}
  Here $B_r(x)$ denotes the ball with centre $x$ and radius $r$.
\end{definition}
We recall that if $u \in W^{1,1} (\Omega),$ then the Lebesgue points of $u$ exist for $\mathcal{H}^{n-1}$-almost everywhere (see Theorem 4.19 and Theorem 5.12 in \cite{EvansGariepy} or Theorem 3.4.2 in \cite{Ziemerbook}).

The following Poincar\'e-Sobolev inequalities, presented in the sense of Choquet with respect to Hausdorff content, directly follow from Theorem 3.7 in \cite{Petteri_2023} through a standard argument of approximating Sobolev functions by smooth functions.

\begin{lemma}\label{poincare}
    Let $0 \leq \sigma <1$, $1 < p <n$, $\beta= n-\sigma p $ and $B$ be an open ball in $\mathbb{R}^n$. If $u \in W^{1,p}(B)$, then there exists a constant $c$ depending only on $n$, $\sigma $ and $p$ such that \begin{align}\label{Poincareduetopetteri}
        \inf\limits_{c \in \mathbb{R}} \left( \int_B | u(x) - c|^{\frac{p \beta}{ n - p}} \; d \mathcal{H}^{\beta}_\infty \right)^{ \frac{n - p}{ p \beta}} \leq C \left( \int_B | \nabla u (x) |^p \; dx \right)^{\frac{1}{p}},
    \end{align}
    where $u $ on the left hand side is understood  as its precise representative. 
\end{lemma}

\begin{proof}
    Suppose that $u \in W^{1,p } (\mathbb{R}^n)$.  Then there exists a sequence of functions $\{u_m\}_{m=1}^\infty \subseteq  W^{1,p } (B) \cap C^\infty (B)$ such that
    $u_m \to u $ in $W^{1,p } (B)$. By choosing a  subsequence if necessary, we may assume that $\Vert u_m - u_{m+1} \Vert_{W^{1,p} (B)} \leq \frac{1}{4^m}$. It follows that 
    \begin{align*}
        u_m(x) \to u(x) \text{ for Lebesgue almost every }  x \in B
    \end{align*}
and
    \begin{align}\label{meanvalueofuminB}
        |(u_{m+1}- u_m)_B| &\leq \frac{1}{|B|} \int_B |u_{m+1} (x) - u_m (x)|\; dx\\ \nonumber
        & = \left( \frac{1}{|B|} \int_B |u_{m+1} (x) - u_m (x)|^p \; dx  \right)^{\frac{1}{p}} \\ \nonumber
        & \leq  \left( \frac{1}{|B|} \right)^{\frac{1}{p}} \frac{1}{4^m}. \nonumber
    \end{align}
Moreover, Theorem 3.7 in \cite{Petteri_2023} yields that for $m \in \mathbb{N}$, there exists a constant $C = C(n,k,p)$ such that
\begin{align}\label{PoincareduetopetteriforCinfinity}
        \left( \int_B | u_m(x) - (u_m)_B|^{\frac{p \beta}{ n - p}} \; d \mathcal{H}^{n - \sigma p}_\infty \right)^{ \frac{n - p}{ p (n - \sigma  p)}} \leq C \left( \int_B | \nabla u_m (x) |^p \; dx \right)^{\frac{1}{p}},
    \end{align}
where 
\begin{align*}
    \left( u_m \right)_B = \frac{1}{|B|} \int_B u(x) \; d x
\end{align*}
denotes the average of $u_m$ over $B.$

The combination of (\ref{meanvalueofuminB}) and (\ref{PoincareduetopetteriforCinfinity}) then gives 
\begin{align}\label{Rapidconvergeofum}
  \nonumber  \Vert u_m - u_{m+1} \Vert_{L^{\frac{p\beta}{n-p}} (B; \mathcal{H}^{\beta}_\infty)} &\leq  \Vert u_m - u_{m+1} -( u_m -u_{m+1})_B \Vert_{L^{\frac{p \beta}{n-p}} (B; \mathcal{H}^{\beta }_\infty)}\\
 \nonumber   & + \Vert ( u_m - u_{m+1} )_B \Vert_{L^{\frac{p \beta}{n-p}} (B; \mathcal{H}^{\beta }_\infty)}\\
 \nonumber   & \leq C \Vert \nabla (u_m -  u_{m+1}) \Vert_{L^p (B)}\\
   \nonumber   &+ \mathcal{H}^{\beta }_\infty (B)^{\frac{n-p}{ p \beta}} \left| (u_m - u_{m+1})_B \right|\\
  & \leq  \frac{C}{4^m} \left( 1+ \mathcal{H}^{\beta }_\infty (B)^{\frac{n-p}{ p \beta}} \frac{1}{|B|^{\frac{1}{p}}}  \right).
\end{align}
By a standard argument of Cauchy in measure (see Proposition 2.5 and Proposition 6.1 in \cite{PonceSpector} for example), there exists a function $\Bar{u} \in L^{\frac{p(\beta)}{n-p}} (B; \mathcal{H}^{\beta}_\infty)$ such that 
\begin{align*}
    u_m \to \Bar{u} \text{ in } L^{\frac{p \beta}{n-p}} (B; \mathcal{H}^{\beta}_\infty)
\end{align*}
and 
\begin{align*}
    u_m (x) \to \Bar{u} (x) \text{ for }\mathcal{H}^{\beta}_\infty \text{ almost every } x.
\end{align*}
Using Fatou's Lemma for the Choquet integral with respect to Hausdorff content (see (1.2) in \cite{PonceSpector} for example), we obtain
\begin{align*}
& \left( \int_B | \Bar{u}(x) - (\Bar{u})_B|^{\frac{p \beta}{ n - p}} \; d \mathcal{H}^{\beta}_\infty \right)^{ \frac{n - p}{ p \beta}}\\
&=    \left( \int_B \lim\limits_{m \to \infty}| \Bar{u}_m(x) - (\Bar{u}_m)_B|^{\frac{p \beta}{ n - p}} \; d \mathcal{H}^{\beta}_\infty \right)^{ \frac{n - p}{ p \beta}}\\
& \leq  \liminf_{m \to \infty}  \left( \int_B | \Bar{u}_m(x) - (\Bar{u}_m)_B|^{\frac{p \beta}{ n - p}} \; d \mathcal{H}^{\beta}_\infty \right)^{ \frac{n - p}{ p \beta}}\\
&\leq \lim\limits_{m \to \infty}C \left(\int_B |\nabla u_m |^p \; dx \right)^{\frac{1}{p}}\\
&=C  \left(\int_B |\nabla u|^p \; dx \right)^{\frac{1}{p}}.
\end{align*}
It is now sufficient to show that the precise representative $u^*$ of $u$ satisfying $u^* = \Bar{u}$ for $\mathcal{H}^{\beta }_\infty$ almost everywhere. Since $\Bar{u}(x) = u(x)$ for Lebesgue almost every $x \in B$, we have $(\Bar{u})_B = u_B$ and  the precise representative of $\Bar{u}$ and $u$ defined as in (\ref{preciserepresentative}) are the same. As $\Bar{u} \in L^{\frac{p(\beta)}{n-p}} (B; \mathcal{H}^{\beta}_\infty)$ implies that the Lebesgue points of $\Bar{u}$ exist $\mathcal{H}^{\beta}_\infty$ almost everywhere in $B$ (see \cite{ChenOoiSpector}*{Corollary 1.3}), we have $ u^* (x) = (\Bar{u})^* (x)= \Bar{u} (x)  $ for $\mathcal{H}^{\beta}_\infty$ almost everywhere in $B$. Thus $u^* = \Bar{u} \in L^p (B; \mathcal{H}^{\beta}_\infty)$.

\end{proof}

We next recall a result which is due to Meyers and Ziemer \cite{Meyers-Ziemer} and Maz'ya \cite{maz_trace} (see also \cite{Adams1988}).

\begin{lemma}\label{AdMaz}
     Let $0\leq k < n\in \mathbb{N}$ and $u \in C_c^k (\mathbb{R}^n)$. Then there exists $C=C(n)$ such that 
\begin{align}\label{Adamsestimate}
    \int_{\mathbb{R}^n} | u| \; d\mathcal{H}^{n-k}_\infty
\leq C \int_{\mathbb{R}^n} | \nabla^k u| \; dx. 
\end{align}

\end{lemma}

We introduce an auxiliary function that will be utilized in the sequel. Let $Q \subset \R^n$, and denote $\ell(Q)$ as the side length of $Q$ and $c_Q$ as the center of $Q$. Let $\phi$ be a function in $C_c^\infty (\R^n)$ satisfying $0 \leq \phi \leq 1$, $\phi(x) = 1$ for all $x = (x_1, x_2, \ldots, x_n)$ satisfying $|x_i|\leq \frac{1}{2}$, and $\phi(x) = 0$ for $x = (x_1, x_2, \ldots, x_n)$ satisfying $|x_i| > 1$. Define the function $\phi_Q$ by $\phi_Q(x) = \phi\left(\frac{x - c_Q}{\ell(Q)}\right)$ for $x \in \R^n$ and $\delta > 0$. It is clear that $\phi_Q \in C_c^\infty(\R^n)$ satisfies $\phi_Q (x) = 1$ for $x \in Q$ and $\phi_Q (x) = 0$ for $x \notin 2Q$. Moreover, for any $k \in \R^n$, there exists a constant $C$ such that
\begin{align}\label{scaleestimate}
    \Vert \nabla^k \phi_Q \Vert_{L^\infty (\R^n)} \leq  \frac{C}{ \ell(Q)^{k}}.
\end{align}

\section{Proofs of the Main Results}

We begin with the

\begin{proof}[Proof of Theorem \ref{BMObetabyPoincare}]
   For any $\sigma  \in (0,1)$ and $p \in (1,n)$ satisfying $\beta = n - \sigma p $, Lemma \ref{poincare} yields that there exists a constant $C $ depending on $n$, $p$ and $\beta$ such that
\begin{align}\label{mainresult2estimate1}
        \inf\limits_{c \in \mathbb{R}} \left( \int_Q | u(x) - c|^{\frac{p \beta}{ n - p}} \; d \mathcal{H}^{\beta}_\infty \right)^{ \frac{n - p}{ p \beta}} \leq C \left( \int_Q | \nabla u (x) |^p \;dx \right)^{\frac{1}{p}}.
    \end{align}
Now choosing $p = n - \varepsilon$ with $\varepsilon = \min (p \beta, n-1)$, we have $\frac{p \beta}{ n-p} >1$.   An application of H\"older's inequality of the Hausdorff content (see \cite{Petteri_2023}*{(C7) on p.~5} for example) and deduce that 
\begin{align}\label{mainresult2estimate2}
    \inf\limits_{c \in \mathbb{R}} \left( \int_Q | u(x) - c| \; d \mathcal{H}^{\beta}_\infty \right)
    & \leq 2\inf\limits_{c \in \mathbb{R}} \left( \int_Q | u(x) - c|^{\frac{p \beta}{ n - p}} \; d \mathcal{H}^{\beta}_\infty \right)^{ \frac{n - p}{ p \beta}} \Vert \chi_Q \Vert_{L^{ \frac{p \beta}{ p \beta -n +p}}(Q;\mathcal{H}^\beta_\infty)} \nonumber \\
    &\leq 2\Vert \nabla u \Vert_{L^p (Q; \mathbb{R}^n)}  \Vert \chi_Q \Vert_{L^{ \frac{p \beta}{ p \beta -n +p}}(Q;\mathcal{H}^\beta_\infty)}.
\end{align}
In particular,  H\"older's inequality in Lorentz scale yields that 
\begin{align}
    \Vert \nabla u \Vert_{L^p(Q; \mathbb{R}^n)} \leq C \Vert \nabla u \Vert_{L^{n, \infty}(Q; \mathbb{R}^n)} \Vert \chi_Q \Vert_{L^{s,1} (Q; \mathbb{R}^n)},
\end{align}
where $\frac{1}{p} = \frac{1}{n}+ \frac{1}{s}$. 
Since
\begin{align*}
    \Vert \chi_Q \Vert_{L^{ \frac{p \beta}{ p \beta -n +p}}(Q;\mathcal{H}^\beta_\infty)} = C \ell(Q)^{\frac{p \beta }{p \beta -n +p}}
\end{align*}
and 
\begin{align*}
    \Vert \chi_Q \Vert_{L^{s,1} (Q; \mathbb{R}^n)} = C \ell(Q)^{\frac{n-p}{p}},
\end{align*}
The combination of (\ref{mainresult2estimate1}) and (\ref{mainresult2estimate2}) then gives that 
\begin{align*}
    \inf\limits_{c \in \mathbb{R}} \int_Q |u(x) -c| \; d \mathcal{H}^\beta_\infty \leq C \ell(Q)^\beta \Vert \nabla u \Vert_{L^{n, \infty} (Q; \mathbb{R}^n)},
\end{align*}
    which completes the proof.
\end{proof}

We next give the

\begin{proof}[Proof of Theorem \ref{MorreyBMO}]
It is sufficient to prove (\ref{Poincare1}) since (\ref{Morreyn1eestimate}) follows immediately by dividing both sides with $r^{n-1}$.  To this end, We first assume that $ u \in W^{1,1} (B(0,1))$. An application of Theorem 5.13.2 in \cite{Ziemerbook} then yields that there exist constants $b$,  $C \in \mathbb{R} $ such that 
\begin{align*}
    \int_{B_1(0)} |u-b| \; d\mu \leq C \int_{B(0,1)} |\nabla u | \; dx 
\end{align*}
holds for all the locally finite Radon measures  $\mu $ satisfying 
\begin{align}\label{ballgrothcondition}
    \mu(B(x,r)) \leq r^{n-1}.
\end{align}
Taking supremum over all the locally finite Radon measures  $\mu $ satisfying (\ref{ballgrothcondition}), we obtain
\begin{align*}
    \int_{B(0,1)} | u -b| \; d\mathcal{H}^{n-1}_\infty \leq C \int_{B(0,1)} |\nabla u| \; dx ,
\end{align*}
by (\ref{HB}).

For general $B= B(x,r) \subseteq \Omega$, let $w(y) = u(x+ry)$. Then $w \in W^{1,1} (B_1(0))$ and thus there exists $b$, $c \in \mathbb{R}$ such that 
\begin{align*}
    \int_{B(0,1)} |u(x+ ry ) -b|\; d\mathcal{H}^{n-1}_\infty \leq C r \int_{B(0,1)} | \nabla u (x +ry)| \; dx.
\end{align*}
By changing of variables we obtain
\begin{align*}
    r^{n-1 } \int_B | u-b |\; d\mathcal{H}^{n-1}_{\infty} \leq C r^{n-1} \int_B |\nabla u |\; dx,
\end{align*}
which upon division of both sides by $r^{n-1}$ yields the inequality \eqref{Poincare1}.
    
\end{proof}

We now give the
\begin{proof}[Proof of Theorem \ref{MorreyBMO2}]
Let $\{ u_m\}_{m=1}^\infty$ be a sequence of functions in $C_c^\infty(\Omega)$ such that $u_m \to u$ in $W^{k,1}_0 (\Omega)$. By choosing a subsequence if necessary, we may assume
\begin{align*}
    \int_\Omega |\nabla^k u_{m+1} - \nabla^k u_m| \; dx \leq \frac{1}{4^m} ,
\end{align*}
which yields that
\begin{align*}
    \int_\Omega |u_{m+1} - u_m| \; d\mathcal{H}^{n-k}_\infty &\leq C \int_\Omega |\nabla^k u_{m+1} - \nabla^k u_m| \; dx \\
    &\leq \frac{C}{4^m}
\end{align*}
by (\ref{Adamsestimate}). This implies that $u_m$ converges to a function $\Bar{u}$ in $L^1(\mathcal{H}^{n-k}_\infty)$ and $u_m(x) \to \Bar{u}(x)$ for $\mathcal{H}^{n-k}$ almost every $x \in \mathbb{R}^n$. Moreover, it follows from a similar argument of the last part of the proof of Lemma \ref{poincare} that $(\Bar{u})_B$ = $u_B$ and the precise representative of $u$ and $\Bar{u}$ are identical in $L^1(\mathcal{H}^{n-k}_\infty)$.

Now fix a cube $Q \subset \mathbb{R}^n$ and $\phi_Q$ be defined as above. We apply Fatou's lemma for the Choquet integral with respect to Hausdorff content (see (1.2) in \cite{PonceSpector} for example) and deduce that 
\begin{align*}
    \int_Q |\Bar{u} - u_{2Q}| \; d\mathcal{H}^{n-k}_\infty & = \int_Q \lim_{m\to \infty} |u_m - u_{2Q}| \; d\mathcal{H}^{n-k}_\infty\\
    &\leq \liminf_{m \to \infty} \int_Q |u_m - u_{2Q}|\; d \mathcal{H}^{n-k}_\infty\\
    &= \liminf_{m \to \infty} \int_Q |(u_m - u_{2Q}) \phi_Q|\; d \mathcal{H}^{n-k}_\infty\\
    &\leq  C \liminf_{m \to \infty} \int_{\mathbb{R}^n} \left| \nabla^k \left[  (u_m  -u_{2Q} ) \phi_Q  \right] \right|\;dx \\
    &=  C \liminf_{m \to \infty} \int_{2Q} \left| \nabla^k \left[  (u_m  -u_{2Q} ) \phi_Q  \right] \right|\;dx,
\end{align*}
where we use Lemma \ref{AdMaz} in the last inequality. Moreover, we note that inequality (\ref{scaleestimate}) gives
\begin{align*}
    \left| \nabla^k \left(  (u_m  -u_{2Q} ) \phi_Q  \right) \right| & \leq |(u_m  -u_{2Q} )| \left| \nabla^{k} \phi_Q \right| +  \sum_{j=1}^k \left| \nabla^{k-j} \phi_Q \right| \left| \nabla^j u_m \right| \\
    & \lesssim |(u_m  -u_{2Q} )|  \ell(Q)^{-k}+ \sum_{j=1}^k \ell(Q)^{j-k} \left| \nabla^j u_m  \right|
\end{align*}
and thus, utilizing the classical Poincar\'e inequality for the first term, we find
\begin{align*}
    \int_Q |\Bar{u} - u_{2Q}| \; d\mathcal{H}^{n-k}_\infty & \leq C \sum_{j=1}^k \ell(Q)^{j-k} \liminf_{m \to \infty} \int_{2Q} \left| \nabla^j u_m  \right| \; dx \\
    & = C \sum_{j=1}^k \ell(Q)^{j-k} \Vert \nabla^j u \Vert_{L^1 (2Q)} \\
    & \leq C \ell(Q)^{n-k} \sum_{j=1}^k \Vert \nabla^j u \Vert_{\mathcal{M}^{n-j}(\Omega)}.
\end{align*}
Now dividing both sides with $\ell(Q)^{n-k}$ in the last estimate, we obtain 
\begin{align}
   \nonumber \frac{1}{\ell(Q)^{n-k}} \int_Q |\Bar{u}-u_Q| \; d \mathcal{H}^{n-k}_\infty \leq C \sum_{j=1}^k \Vert \nabla^j u \Vert_{\mathcal{M}^{n-j}(\Omega)}
\end{align}
and (\ref{BMO2estimate}) follows by taking the supremum over all cubes $Q \subseteq \Omega$. This completes the proof.
\end{proof}

We next prove Theorems \ref{estimatefornonlinearpLaplacian} and \ref{estimatefornonlinearpLaplacian_A}.

\begin{proof}[Proof of Theorems \ref{estimatefornonlinearpLaplacian} and \ref{estimatefornonlinearpLaplacian_A}]
Let $\{F_m\}$ be a smooth sequence for which $\operatorname*{div}F_m=0$, $F_m \to F$ in $L^1(\Omega;\mathbb{R}^n)$, and
\begin{align}\label{mollification_noincrease}
\lim_{m \to \infty} \|F_m\|_{L^1(\Omega;\mathbb{R}^n)} =\|F\|_{L^1(\Omega;\mathbb{R}^n)},
\end{align}
(see e.g. \cite[Lemma 3.3]{FG}).  Then the existence results in \cite{DHM,DHM1} (which uses an approximation by mollification) yields that there exist functions $U_m \in BMO(\Omega; \mathbb{R}^n) \cap W^{1,n}_0 (\Omega; \mathbb{R}^n)$ which satisfy
\begin{align*}
\int_{\Omega} |\nabla U_m|^{n-2}\nabla U_m \cdot \nabla \Phi\;dx = \int_{\Omega} F_m \cdot \Phi \;dx,
\end{align*}
for any $\Phi \in W^{1,n}(\mathbb{R}^n; \mathbb{R}^n)$.  The fact that $\nabla U_m \in L^n(\Omega;\mathbb{R}^n)$ allows us to take $\Phi=U_m$ in the preceding inequality to obtain
\begin{align*}
\int_{\Omega} |\nabla U_m|^{n}\;dx = \int_{\Omega} F_m \cdot U_m \;dx,
\end{align*}

By Lemma 3.4 in \cite{BVS}, we have
\begin{align*}
\left|\int_{\Omega} F_m \cdot U_m \;dx \right| \leq C \|F_m\|_{L^1(\Omega)} \|\nabla U_m \|_{L^n(\Omega;\mathbb{R}^{n\times n})},
\end{align*}
and therefore
\begin{align*}
\int_{\Omega} |\nabla U_m|^{n}\;dx \leq C \|F_m\|_{L^1(\Omega;\mathbb{R}^n)} \|\nabla U_m \|_{L^n(\Omega;\mathbb{R}^{n\times n})}.
\end{align*}
Together with the inequality \eqref{mollification_noincrease} this implies
\begin{align*}
\|\nabla U_m \|_{L^n(\Omega;\mathbb{R}^{n\times n})} &\leq C^{1/(n-1)} \|F_m\|^{1/(n-1)}_{L^1(\Omega;\mathbb{R}^n)} \\
&\leq C^{1/(n-1)} \|F\|^{1/(n-1)}_{L^1(\Omega;\mathbb{R}^n)}.
\end{align*}
This shows $\{\nabla U_m\}_m$ is a bounded sequence in $L^n(\Omega)$, so that up to a subsequence it converges to $\nabla U \in L^n(\Omega)$, while lower-semicontinuity of the norm with respect to weak-convergence yields the estimate
\begin{align*}
\|\nabla U \|_{L^n(\Omega;\mathbb{R}^{n\times n})} \leq C^{1/(n-1)} \|F\|^{1/(n-1)}_{L^1(\Omega;\mathbb{R}^n)}.
\end{align*}
\end{proof}
We conclude the paper with the 
\begin{proof}[Proof of Theorem \ref{estimatefornonlinearpLaplacian_Lorentz}]
Suppose $F \in L^1(\Omega;\mathbb{R}^n)$ is supported in $\Omega$ and $\operatorname*{div}F=0$.  Set
\begin{align*}
G_i:= -\nabla (-\Delta)^{-1}F_i.
\end{align*}
Then we have
\begin{align*}
\operatorname*{div} G_i = F_i
\end{align*}
in the sense of distributions, while estimates for singular integrals yield
\begin{align*}
    \|G_i\|_{L^{n/(n-1),1}(\mathbb{R}^n;\mathbb{R}^n)} &\lesssim \|I_1F\|_{L^{n/(n-1),1}(\mathbb{R}^n;\mathbb{R}^n)}.
\end{align*}
The fact that $\operatorname*{div}F=0$ allows one to utilize the endpoint estimate \cite[Theorem 1.1]{HS} (see also \cite{HRS}) 
\begin{align*}
\|I_1F\|_{L^{n/(n-1),1}(\mathbb{R}^n;\mathbb{R}^n)} \lesssim\|F\|_{L^1(\mathbb{R}^n;\mathbb{R}^n)}
\end{align*}
to conclude that
\begin{align*}
    \|G_i\|_{L^{n/(n-1),1}(\mathbb{R}^n;\mathbb{R}^n)} &\lesssim \|F\|_{L^1(\mathbb{R}^n;\mathbb{R}^n)}.
\end{align*}
The claim then follows by Martino and Schikorra's \cite[Theorem 1.1]{MS2024} estimate
\begin{align}
\|\nabla U \|_{L^{n,n-1}(B_\theta;\mathbb{R}^{n\times n})}^{n-1} \lesssim \|G\| _{L^{n/(n-1),1}(B_1;\mathbb{R}^n)} + \|\nabla U \|_{L^{n-\epsilon}(B_1;\mathbb{R}^{n\times n})}^{n-1}
\end{align}
and the preceding theorem to absorb the second term on the right-hand-side.
\end{proof}

\section*{Acknowledgments}
Y.-W. Chen is supported by the National Science and Technology Council of Taiwan under research grant number 113-2811-M-002-027.
J. Manfredi is supported by the Simons Foundation under Gift Number 962828. 
D. Spector is supported by the National Science and Technology Council of Taiwan under research grant numbers 110-2115-M-003-020-MY3/113-2115-M-003-017-MY3 and the Taiwan Ministry of Education under the Yushan Fellow Program.

The authors would like to thank Armin Schikorra and Rupert Frank for their comments which have improved the clarity and presentation of the manuscript.

\end{document}